\documentclass[12pt,oneside]{amsart}
\usepackage{amscd,amsmath,amssymb,amsfonts,a4}
\usepackage{hyperref}

\newlength{\rulebreite}

\def\timesover#1#2#3{\ \xymatrix@1@=0pt@M=0pt{ _{#1}&\times&_{#2} \\& ^{#3}&}\ }
\def\otimesover#1#2#3{\ \xymatrix@1@=0pt@M=0pt{ _{#1}&\otimes&_{#2} \\& ^{#3}&}\ }

\usepackage[all]{xy}
\theoremstyle{plain}
\newtheorem{thm}{Theorem}

\newtheorem{cor}[thm]{Corollary}
\newtheorem{prop}[thm]{Proposition}
\theoremstyle{definition}

\newtheorem{questions}[thm]{Questions}

\numberwithin{thm}{section}
\numberwithin{equation}{section}

\newcommand{\ga}[2]{\begin{gather}\label{#1}#2 \end{gather}}

\newcommand{\Pic}{{\rm Pic}}

\newcommand{\Spec}{{\rm Spec \,}}

\newcommand{\sD}{{\mathcal D}}

\newcommand{\sO}{{\mathcal O}}

\newcommand{\A}{{\mathbb A}}

\newcommand{\C}{{\mathbb C}}

\newcommand{\E}{{\mathbb E}}
\newcommand{\F}{{\mathbb F}}
\newcommand{\G}{{\mathbb G}}

\newcommand{\N}{{\mathbb N}}
\renewcommand{\P}{{\mathbb P}}
\newcommand{\Q}{{\mathbb Q}}

\newcommand{\Z}{{\mathbb Z}}

\begin{document}

\title[Flat bundles]{On flat bundles in characteristic $0$ and $p>0$}
\author{H\'el\`ene Esnault}
\address{
Universit\"at Duisburg-Essen, Mathematik, 45117 Essen, Germany}
\email{esnault@uni-due.de}
\date{ May 14, 2012}
\thanks{The author is supported by   the SFB/TR45 and the ERC Advanced
Grant 226257}
\begin{abstract} We discuss analogies between the fundamental groups of flat bundles in characteristic $0$ and $p>0$.
\end{abstract}
\maketitle
\section{Introduction}
In this short note, we discuss some analogies between $\sO_X$-coherent $\sD_X$-modules 
over $X$ quasi-projective over $k=\C$, the field of complex numbers, and over an algebraically closed field of characteristic $p>0$. For the finiteness problems we singled out, they are striking. Yet we have no way to go from characteristic $0$ to characteristic $p>0$ and vice-versa. 

\medskip

The note relies on work I have been doing in the last years with numerous mathematicians, notably  Ph\`ung H\^o Hai, Adrian Langer, Vikram Mehta, Xiaotao Sun. I thank them all for the fruitful exchange of ideas and the pleasure of sharing them. 
 I thank Alexander Beilinson and Bruno Klingler for discussions on
questions related to the topic of this note, 
 Lars Kindler for all the discussions on regular singularities reflected in his work \cite{Kl}.  In addition, I thank
 Moritz Kerz for the mathematics I learned from him.

\section{Characteristic $0$} \label{sec:0}
Let $X$ be a smooth connected  quasi-projective variety defined over the field $\C$ of complex numbers, together with a  complex point $a \in X(\C)$. Then one has the topological fundamental group $\pi_1^{\rm top}(X,a)$ of loops centered at $a$  of the underlying topological spaces, modulo homotopy,
 as defined by Poincar\'e (\cite{Po}). 
 
 \medskip

Grothendieck  developed a different viewpoint on it. He considered the category 
of topological covers of $X$. The objects are topological covers $\pi: Y\to X$
while the maps $\theta: \pi_1\to \pi_2$ are continuous maps $Y_1\to Y_2$ over
$X$. The point $a$ defines a fiber functor  $\omega^{\rm top}_a$  to the
category of sets by sending $\pi$ to $\pi^{-1}(a)$ and $\theta$ to $\theta|_a:
\pi_1^{-1}(a)\to \pi_2^{-1}(a)$. Then the fundamental group becomes identified
with the automorphism group of $\omega^{\rm top}_a$: $\pi_1^{\rm
top}(X,a)\xrightarrow{\cong}{\rm Aut}(\omega^{\rm top}_a)$, by sending a loop
centered at $a$ to the collection of induced automorphisms of
 $\pi^{-1}(a)$ for all topological covers $\pi$ (\cite[10.11]{Del2}).
 
 \medskip

He restricted $\omega^{\rm top}_a$ to
the full subcategory  of finite covers of $X$. By the 
 Riemann existence theorem,   the natural functor from  
 the category of finite \'etale covers of $X$ to the category of finite topological covers is
 an equivalence of categories.
 Restricting $\omega_a^{\rm top}$ to it defines the fiber functor 
 $\omega_a$ from the category of finite \'etale covers of $X$ to finite sets.
  This defines the \'etale fundamental group $\pi_1^{{\rm \acute{e}t}}(X,a)$ centered at $a$ as the automorphism group of $\omega_a$.
 This is  thus the profinite completion of $\pi_1^{\rm top}(X,a)$. 
Grothendieck's definition has the advantage to be general. He defines  the
\'etale fundamental group $\pi_1^{\rm \acute{e}t}(X,a)$ in \cite[V]{SGA1}   for
any connected normal scheme $X$ 
 exactly in the same way, in particular for $X=\Spec k$, where $k$ is a field, in which case  $\pi_1^{{\rm \acute{e}t}}(X,a)$ becomes identified with 
 ${\rm Gal}(\bar k/k)$ where $\bar k \subset \kappa(a)$ is the separable closure of $k$ lying in the residue field of the chosen geometric point.  
 The fundamental theorem in those theories says that the category of topological (resp. finite \'etale) covers of $X$ 
is equivalent, via the fiber functor $\omega_a^{\rm top}$ (resp. $\omega_a$), to
the representation category of $\pi_1^{\rm top}(X,a)$  (resp. $\pi_1^{\rm
\acute{e}t}(X,a)$) in the category of sets (resp. finite sets).
(\cite{Del2}, {\it loc.~cit.} \cite[V,~Proposition~5.8]{SGA1}).
 
 \medskip
 
Coming back to $X$ a smooth connected  quasi-projective variety over  $\C$, the underlying topological space is a finite CW complex (\cite[Remark~p.40]{AH}), and thus 
$\pi_1^{\rm top}(X,a)$ is finitely presented, in particular finitely generated. This implies that its pro-algebraic completion 
\ga{}{\pi_1^{\rm top}(X,a)^{{\rm alg}}:=\varprojlim_H H\notag}
where $H\subset GL(r, \C)$ is the Zariski closure of a complex linear representation of $\pi_1^{\rm top}(X,a)$, is controlled by the pro-finite completion $\pi_1^{\rm \acute{e}t}(X,a)$. For example, if $f:X\to Y, \ a \mapsto b$ is a morphism of smooth complex connected quasi-projective varieties, then 
if $f_*: \pi_1^{\rm \acute{e}t}(X,a)\to \pi_1^{\rm \acute{e}t}(Y,b)$ is an isomorphism, so is $f_*: \pi_1^{\rm top}(X,a)^{{\rm alg}}\to \pi_1^{\rm top}(Y,b)^{{\rm alg}}$ (Malcev \cite{Mal},  Grothendieck \cite[Th\'eor\`eme~1.2]{Gr}).
In particular, applied to $Y=\Spec \C$, it says that if the \'etale fundamental group is trivial, there are no non-trivial complex linear representations of $\pi_1^{\rm top}(X,a)$, that is no non-trivial complex linear systems. 
The proof of this last point is easy. Since $\pi_1^{\rm top}(X,a)$ is finitely
generated, a complex linear representation $\rho: \pi_1^{\rm top}(X,a) \to
GL(r, \C)$ has values in $GL(r, A)$ where $A$ is a ring of finite type over
$\Z[\frac{1}{N}]$, for some natural number $N\neq 0$. If $\rho$ is not trivial, then there is a
 closed point $s\in \Spec A$, with finite residue field $\kappa(s)$, such
that the specialization $\rho\otimes \kappa(s): \pi_1^{\rm top}(X,a) \to GL(r,
\kappa(s))$ is not trivial as well.
\medskip

Let us write the three groups in a diagram. One has 
\ga{}{\xymatrix{\pi_1(X,a)^{{\rm top}} \ar[d] \ar[dr] \\
\pi^{\rm top}_1(X,a)^{\rm alg} \ar[r] & \pi_1^{ \rm \acute{e}t  }(X,a)
}\label{triangle} }
where the top group is a finitely presented  abstract group, the left bottom group is a  pro-algebraic group over $\C$, the right bottom group is a pro-finite group,  which is the pro-finite completion of the top one, and the horizontal 
map is a morphism of pro-algebraic groups over $\C$ when one thinks of the bottom right group as a constant pro-algebraic group over $\C$. 
The horizontal map is surjective as any finite \'etale cover $\pi: Y\to X$
defines the local system $\pi_*\C$ with finite monodromy. In fact, the
horizontal map is the pro-finite completion map. We saw that in the sense
discussed above, $\pi_1^{ \rm \acute{e}t  }(X,a)$ controls $\pi^{\rm
top}_1(X,a)^{\rm alg}$.  So $\pi_1^{\rm top}(X,a)^{{\rm alg}}$
is controlled by its profinite completion.

\medskip

On the other hand, by the Riemann-Hilbert correspondence \cite{Del}, the representation category of $\pi_1^{\rm top}(X,a)^{{\rm alg}}$ in finite dimensional complex vector spaces is tensor equivalent to the category of vector bundles with a regular singular flat connection. 
The latter is equivalent to the category of regular singular $\sO_X$-coherent $\sD_X$-modules.
This notion is purely algebraic in the following sense: if $k$ is a  field of characteristic $0$, $X$ is a smooth geometrically connected quasi-projective variety defined over $k$, then the category of vector bundles  $(E,\nabla)$ with a flat connection, or equivalently the category
of $\sO_X$-coherent $\sD_X$-modules, 
 is an abelian, $k$-linear, rigid tensor category, neutralized by the choice of a $k$-point $a\in X(k)$. The fiber functor $\omega^{\rm alg}_a$ sends $(E,\nabla)$ to $E|_a$. The Tannaka group 
$\pi^{\rm alg}_1(X,a):={\rm Aut}^{\otimes}(\omega^{\rm alg}_a)$ is
then a pro-algebraic group over $k$, its representation category in finite
dimensional $k$ vector spaces is via $\omega^{{\rm alg}}_a$ equivalent to the
category of vector bundles with a  flat connection.  The category of flat
connections contains the full subcategory of flat connections with regular
singularities at infinity \cite{Del}. Denoting by $\omega_a^{\rm alg, rs}$ the
restriction of $\omega^{\rm alg}_a$ to this subcategory defines the Tannaka
group $\pi^{\rm alg,rs}_1(X,a):={\rm Aut}^{\otimes}(\omega^{{\rm alg, rs}}_a)$
as a quotient of $\pi_1^{\rm alg}(X,a)$.   If $k=\C$, then $\pi_1^{\rm top}(X,a)^{\rm alg}=\pi_1^{\rm alg, rs}(X,a)$.

 \medskip 
 
Again considering $\pi^{\rm \acute{e}t}_1(X,a)$
as a constant pro-algebraic group over $k$, the homomorphism
\ga{}{\label{pf}
\pi_1^{\rm alg}(X,a)\to \pi^{\rm \acute{e}t}_1(X,a)}
factors through
\ga{}{\pi^{\rm alg}_1(X,a)\to \pi^{\rm alg,rs}_1(X,a) \to  \pi^{\rm
\acute{e}t}_1(X,a)  \label{fact}  }
and $\pi^{\rm
\acute{e}t}_1(X,a)$ is the pro-finite completion both of $\pi_1^{\rm alg}(X,a)$
and  $\pi^{\rm alg,rs}_1(X,a)$.

\medskip 

 However, if $K\supset k$ is a field extension, the natural base change morphisms
 $$\pi_1^{\rm alg}(X\otimes_k K ,a \otimes_k K) \to \pi_1^{\rm alg}(X,a)\otimes_k K$$ and 
$$\pi_1^{\rm alg, rs}(X\otimes_k K ,a \otimes_k K) \to \pi_1^{\rm alg,rs }(X,a)\otimes_k K$$
 of pro-algebraic groups over $K$ are not isomorphisms (\cite[10.35]{Del2}). 
 It is discussed in {\it loc.~cit.} over $X=\G_m$, but this is still not an
isomorphism even if $X$ is projective,  in which case the surjective
homomorphism $ \pi_1^{\rm alg}(X,a)\to \pi_1^{\rm alg,rs }(X,a)$ is an
isomorphism. For example, assume $K=\overline{ k(t)}, \ H^0(X, \Omega^1_X)\neq
0$.  Then the flat connection 
 $(\sO_X, d+t\alpha)$, for $0\neq \alpha\in H^0(X, \Omega^1_X)$ can not be a subquotient of a connection defined over $k$.
 Still the base change morphisms are faithfully flat,  as any subconnection $(E',\nabla')\subset (E,\nabla)\otimes_k K$  is defined over a finite extension of $L\supset k$ in $K$ (\cite[Proposition~2.21]{DM}).
 
 \begin{prop}[\cite{Gr},~Th\'eor\`eme~1.2 over $\C$]  \label{f}
 Let $k$ be an algebraically closed field of characteristic $0$.
 \begin{itemize}
 \item[i)] If  $f: X\to Y$ is a morphism of smooth connected quasi-projective
varieties mapping $a\in X(k)$ to $b\in Y(k)$ . If $f_*:  
\pi_1^{\rm \acute{e}t}(X,a)\to \pi_1^{\rm \acute{e}t}(Y,b)$ is an isomorphism,
then 
 $f_*: \pi_1^{\rm alg, rs }(X,a)\to \pi_1^{\rm alg, rs}(Y,b)$ is an isomorphism
as well.
 \item[ii)] If $X$ is a  smooth connected  quasi-projective variety with $a\in X(k)$, such that  $\pi_1^{\rm \acute{e}t}(X,a)$ is trivial. Then $\pi_1^{\rm alg,rs }(X,a)$ is trivial as well.
 \end{itemize}

 \end{prop}
 \begin{proof}
  Then if  $k\subset K $ is an extension of algebraically closed fields, the base change morphism
 \ga{}{ \pi_1^{\rm \acute{e}t}(X\otimes_k K ,
a\otimes_k K)\to \pi_1^{\rm \acute{e}t}(X,a) \label{bc} }
is an isomorphism  (\cite[Introduction]{LSe}).  

We show i).
Assuming $k$ is embeddable in $\C$, the assumption implies that 
 $\pi_1^{\rm alg,rs }(X\otimes_k\C ,a\otimes_k \C)\to \pi_1^{\rm
alg,rs}(Y\otimes_k \C,b\otimes_k \C)$ is an isomorphism by Malcev-Grothendieck's
theorem. 
 So given $M$ a flat regular singular connection on $Y$, and $N\subset f^*M$ a
subconnection (thus regular singular) on 
$X$, there is a subconnection $N'\subset M\otimes_k \C$ (thus regular singular,
as $\otimes \C$ preserves the property) with $(f\otimes_k
\C)^*(N')=N\otimes_k \C$. These relations are   defined over an affine variety
$S$ over $k$ with $k(S)\subset \C$. Thus if $p^Z_1: Z\times_k S\to Z$  denotes
the first projection, one obtains $N'_S\subset( p^Y_1)^*M$ with $(f\times
1_S)^*(N'_S)=(p^X_1)^* N$ for some relative  flat connection $N_S$. Choosing
$s\in S(k)$ a rational point, the restriction $N'_s$ of $N'_S$ to $Y\times_k s$
fulfills $N'_s\subset M$ (thus is regular singular) and $f^*(N'_s)=N$. This
shows that $\pi_1^{\rm alg, rs
}(X,a)\to \pi_1^{\rm alg, rs}(Y,b)$ is faithfully flat
(\cite[Proposition~2.21]{DM}). 

 Similarly, if $M$ is a flat regular singular connection on $X$,
there is a flat regular singular connection $N$ on $Y\otimes_k \C$ such that $M$
is a subquotient
of $(f\otimes \C)^*(N)$. As $N$ is regular singular, there is a
smooth compactification $\hat{Y}$ of $Y$ such that $\hat{Y}$ is projective and 
$D= \hat{Y}\setminus Y$ is a normal crossings divisor, and there is a locally
free extension $\hat N$ of $N$ such that the connection extends to $\hat{N}\to
\Omega^1_{\hat Y\otimes_k \C/\C}(\log D)\otimes \hat N$.
Again spreading out,
$\hat N$ is obtained by base change from a flat connection connection 
$\hat N_S \to \Omega^1_{\hat Y\times_k S/S}(\log D)\otimes \hat N_S$ 
on $Y\times_k S$ relative to $S$. 
So restricting $N_S =\hat N_S|_{Y\times_k S} $ to a $k$ point of $S$
shows that $M$ is a subquotient of a flat connection $f^*N_0$, with $N_0$ a
flat regular singular connection defined over $Y$.
This shows that  $\pi_1^{\rm alg, rs }(X,a)\to \pi_1^{\rm
alg, rs}(Y,b)$  is a closed embedding (\cite{DM}, {\it loc.~cit.}). This
finishes
the proof if $k$ is embeddable in $\C$.

In general,  an object or  a morphism between two objects  is defined over a field  $K_0$ of finite type over $\Q$ in $k$ containing the field of definition $k_0$ of $f$, so $f=f_0\otimes_{k_0} k$.  One applies the previous isomorphism to an algebraic closure $\bar K_0$ of $K_0$ to conclude that over $\bar K_0$,  $(f_0\otimes_{k_0} \bar K_0)^*$ identifies objects and morphisms. This finishes the proof of i).

 ii)  is a particular case of i) for $Y=\Spec k$.

 \end{proof}

\medskip

Proposition~\ref{f} i), while applied to $f$ being the Albanese map, implies

\begin{cor}[see \cite{G},~Theorem~0.4 over $\C$] \label{cor:f}
 Let $k$ be an algebraically closed field of characteristic $0$, let $X$ be a
smooth projective variety over $k$. Then $\pi_1^{\rm \acute{e}t}(X,a)$ is
abelian, if and only if every irreducible bundle with a flat connection has
rank $1$. 
\end{cor}

\medskip

Proposition \ref{f} is an algebraic statement.  Its proof relies on the theorem of Malcev-Grothendieck, which in turn is a consequence of $\pi_1^{\rm top}(X,a)$ being finitely generated for $X$ defined over $\C$.   We are not aware of the existence of an algebraic proof to it.  

\medskip

To close this section, we observe that the finite generation of $\pi_1^{\rm
top}(X,a)$  is also reflected, for $X$ projective smooth over $\C$, by the
moduli theory of Simpson \cite{Simp}. The Betti moduli space $M_B(X)$, which
parametrizes complex local systems of rank $r$, is a complex affine variety. He
constructed the de Rham moduli space $M_{dR}(X)$, and, fixing a polarization,  the Higgs moduli space
$M_{{\rm Higgs}}(X)$. Both are quasi-projective varieties. They map to the moduli spaces of semi-stable bundles, so his construction via geometric invariant theory generalizes the classical one
for vector bundles. 

\medskip

Furthermore, they  are all homeomorphic, in fact even real analytically isomomorphic. So for example,  if
one knows that $M_{{\rm
Higgs}}(X)$ is projective, then necessarily it is $0$-dimensional, that is
there are finitely many isomorphism classes of irreducible rank $r$ local
systems (see \cite{Kl} where this argument is used). But to conclude directly
Proposition~\ref{f} ii), one would need for instance that points in $M_{dR}(X)$
say, which correspond to local systems with finite monodromy, are dense. It is true
rank $1$. In higher rank there are rigid local systems which are isolated  (we thank Burt Totaro for the reference \cite{CS}) .  More generally, it would be of high interest to understand the points 
in the three moduli spaces which, in $M_{{\rm Betti}}(X)$, correspond to local systems with finite monodromy.
Those flat connections $(E,\nabla)$ are uniquely determined by the underlying vector bundle $E$. By
 Nori \cite{No}, such vector bundles (which he called ``finite'') are purely algebraically described. They are those vector bundles  $E$ with satisfy a  polynomial equation $f(E)=g(E)$,  where $f,g\in \N[T], f\neq g$. (Here in characteristic $0$ the category is semi-simple, so there is no need to introduce the essentially finite bundles \cite[Section~2]{EH}).  However, those equations involve higher rank bundles as well, so it is difficult to cut them down to $M_{dR}(X)$.

\section{Characteristic $p>0$} \label{sec:p}
Let $X$ be a smooth connected  quasi-projective variety defined over an
algebraically closed field $k$ of characteristic $p>0$, endowed with a point
$a\in X(k)$. 
Clearly we do not have a notion of topological fundamental group at disposal.
But we have the theory of \'etale fundamental groups.

\medskip

 We denote by $X^{(1)}$ the pull-back of $X$ over the Frobenius of
$\Spec k$, and by $F_{X/k}: X\to X^{(1)}$ the Frobenius of $X$ relative to $k$. 
 The Homs in the category of bundles with a flat
 connection  are linear over $\sO_{X^{(1)}}$, so, unless $X$ is proper, there
are not finite dimensional $k$-vector spaces. On the other hand, the category
of 
$\sO_X$-coherent $\sD_X:=\sD_{X/k}$-modules is, as in characteristic $0$, an abelian,
rigid tensor category \cite{G}. The rigidity  comes from the fact that if $E$
is a $\sO_X$-coherent $\sD_{X}$-module, then it is locally free (Katz,
\cite[Theorem~1.3,~Proof]{G}). Fixing a rational point $a\in X(k)$ defines a
fiber functor $\omega_a$ to the category of finite dimensional  $k$-vector
spaces, by assigning to $E$ its restriction in $a$. This defines a pro-algebraic group scheme $\pi_1^{{\rm alg}}(X,a):={\rm
Aut}^{\otimes }(\omega_a)$.

\medskip

Recall that the Riemann-Hilbert correspondence over $\C$ is between
$\sO_X$-coherent regular singular  $\sD_X$-modules
and
local systems, that is solutions,  in the analytic topology,  of
the linear differential equation in the Zariski topology attached to the 
$\sD_X$-module. 
Katz in \cite[Theorem~1.3]{G} showed an analog statement to the
Riemann-Hilbert correspondence in characteristic $p>0$. As a consequence of
 Cartier's characterization of $p$-curvature $0$ connections (\cite[Theorem~5.1]{Ka}), if $E$ is a coherent $\sD_X$-module, 
then the associated
flat connection is spanned by flat sections, now in the Zariski topology. This
defines a bundle $E^{(1)}$ on $X^{(1)}$ together with an isomorphism
$(E,\nabla)\cong (F^*E^{(1)}, d\otimes 1_{E^{(1)}})$, where we write $F_{X/k}^*E^{(1)}=\sO_X\otimes_{F_{X/k}^{-1}\sO_{X^{(1)}}} F_{X/k}^{-1}E^{(1)}$. Then $E^{(1)}$ is precisely the subsheaf of $E$ annihilated by all the differential operators of order $\le (p-1)$ and is a vector bundle on $X^{(1)}$.  In particular, the restriction  $\sigma_0: E\to F^*_{X/k} E^{(1)}$ to the bundles of the isomorphism of flat connections determines uniquely the whole isomorphism. 
Further, 
the differential operators of $X$ of order $\ge p$ act on $E^{(1)}$. The
subsheaf of sections annihilated by all differential operators of order $\le
(p^2-1)$ is a vector bundle $E^{(2)}$ on $X^{(2)}$ and one has an isomorphism of
vector bundles $\sigma_1: E^{(1)}\to F_{X^{(1)}/k}^* E^{(2)}$ etc. 
A stratified bundle $\E=(E^{(n)}, \sigma_n)_{n\in \N}$ is a sequence of bundles $(E^{(0)}=E,E^{(1)}, E^{(2)}, \ldots)$ together with a sequence of $\sO_{X^{(n)}}$-isomorphisms
$\sigma_n: E^{(n)} \to F_{X^{(n)}/k}^* E^{(n+1)}$ of bundles. A morphism $\varphi: \E\to \E'$  is a collection $(\varphi_0,\varphi_1,\varphi_2 ,\ldots)$ where
$\varphi_n: E^{(n)}\to E^{' (n)}$ is a bundle map commuting with the $\sigma_i$ and $\sigma'_i$.  
Katz' theorem ({\it loc.~cit.}) asserts that  the functor  which assigns to a
$\sO_X$-coherent $\sD_X$-module its underlying stratified bundle as explained
above is an equivalence of categories.  This is analog to the Riemann-Hilbert
correspondence. However, as it does not involve a stronger topology than the
Zariski one, there is no growth condition at infinity of solutions which appears
here in the non-proper case. This notion enters later in the sequel. 

\medskip

If $\pi: Y\to X$ is a finite \'etale cover, then $\pi_*\sO_Y$ is an
$\sO_X$-coherent $\sD_X$-module. This defines a surjective homomorphism
\ga{}{ \label{pfp}\pi_1^{\rm alg}(X,a)\to \pi^{\rm \acute{e}t}(X,a)}
of pro-algebraic groups, if one considers $\pi^{\rm \acute{e}t}(X,a)$ as a constant pro-finite algebraic group over $k$.  By \cite[Proposition~13]{dS}, this map is the pro-finite completion, as in \eqref{pf}.

\medskip 

The analogy with the characteristic $0$ theory becomes more involved 
when
one discusses \eqref{fact}.  Finite \'etale tame covers (see \cite{KS} for a
precise and detailed account) define a full subcategory of the category of
finite \'etale covers.  This defines the quotient $\pi_1^{\rm
\acute{e}t,tame}(X,a)$ of $\pi_1^{\rm \acute{e}t}(X,a)$.
On the other hand, Gieseker \cite[Section~3]{G} defines regular singular
$\sO_X$-coherent $\sD_X$-modules, assuming $X$ admits a smooth projective
compactification $\hat{X}$ such that $\hat{X}\setminus X$ is a strict normal
crossings divisor. In \cite[Section~3]{Ki}, the concept of a regular singular $\sO_X$-coherent $\sD_X$-module  is defined unconditionally. 
If $X\subset \hat{X}$ is a partial (i.~e. is not necessarily projective) smooth compactification such that $\hat{X}\setminus X$ is a strict normal crossings divisor,  regularity with respect to this partial compactification is defined as usual,  by assuming the existence of a $\sO_{\hat X}$-coherent extension of the underlying $\sO_X$-locally free sheaf,  on which the $\sD_X$-action extends to an action of the differential operators which stabilize the ideal sheaf of $\hat{X}\setminus X$.  Then the $\sD_X$-module is regular singular if it is relatively to all such partial compactifications. This notion coincides with Gieseker's one if one has a good compactificaton.
The
full subcategory of $\sO_X$-coherent regular singular $\sD_X$-modules is a sub
Tannaka category.   This defines a quotient $\pi_1^{\rm alg, rs}(X,a)$ of
$\pi_1^{\rm alg}(X,a)$. 
The main theorem of \cite[Section~4]{Ki} asserts that the composite of \eqref{pfp} with  $ \pi_1^{\rm \acute{e}t}(X,a) \to \pi_1^{\rm \acute{e}t,tame}(X,a) $ factors through $\pi_1^{\rm alg, rs}(X,a)$ 
\ga{}{\label{factp}  \xymatrix{ \ar[d] \pi_1^{\rm alg}(X,a)\ar[r] & \ar[d] \pi^{\rm \acute{e}t}(X,a)\\
\pi_1^{\rm alg, rs}(X,a)\ar[r] &  \pi^{\rm \acute{e}t, tame}(X,a)
}
}
and that $  \pi^{\rm \acute{e}t, tame}(X,a)$, while considered as a constant pro-algebraic group over $k$,  is the pro-finite completion of $\pi_1^{\rm alg, rs}(X,a)$.

\medskip

Let $k$ be  an algebraically closed field of characteristic $p>0$.
So we may raise the following questions in analogy with Proposition~\ref{f}.
\begin{questions} \label{q:tame}
\begin{itemize}
\item[i)] 
Let $f: X\to Y$ be a morphism between smooth quasi-projective varieties mapping $a\in X(k)$ to $b\in Y(k)$. 
Assume $$f_*:  
\pi_1^{\rm \acute{e}t, tame}(X,a)\to \pi_1^{\rm \acute{e}t, tame}(Y,b)$$ is an isomorphism, then
is 
$$f_*: \pi_1^{\rm alg, rs }(X,a)\to \pi_1^{\rm alg, rs}(Y,b)$$  an isomorphism as well?
 \item[ii)] Let $X$ be a  smooth connected  quasi-projective variety with $a\in
X(k)$, such that  $\pi_1^{\rm \acute{e}t, tame}(X,a)$ is trivial. Then is
$\pi_1^{\rm alg,rs }(X,a)$ trivial as well?
\end{itemize}
\end{questions}

The theory of tame fundamental groups, unfortunately, isn't well developed. 
Natural properties are now yet known. For example, as far as we are aware of, 
K\"unneth formula and the base change property are not known. They  are 
both
known for
the maximal prime to $p$ quotient of $\pi_1^{\rm \acute{e}t}(X,a)$ (see
\cite{Or} and in it Remarque~5.3 for base change).  
 Topological finite presentation or even topological finite generation are not
known
 (see \cite[Th\'eor\`eme~6.1]{OV} for finite generation and for $X$ being the
complement of a  divisor in a smooth projective curve over an algebraically
closed field $k$, but in higher dimension, we do not know the necessary
Lefschetz theorems to conclude).
A main obstruction  to generalize the corresponding results in \cite{SGA1}  is the absence of resolution of singularities (see for example 
\cite[Theorem~A.15]{LO} for the homotopy exact sequence for smooth varieties
over an algebraically closed field for the  tame fundamental groups under the
assumption of  resolution of singularities).

\medskip

Of course we can raise the same questions dropping the tameness assumption.

\begin{questions} \label{q:proj}
\begin{itemize}
\item[i)] 
Let $f: X\to Y$ be a morphism between smooth quasi-projective varieties mapping $a\in X(k)$ to $b\in Y(k)$. 
Assume $$f_*:  
\pi_1^{\rm \acute{e}t}(X,a)\to \pi_1^{\rm \acute{e}t}(Y,b)$$ is an isomorphism, then
is 
$$f_*: \pi_1^{\rm alg }(X,a)\to \pi_1^{\rm alg}(Y,b)$$  an isomorphism as well?
 \item[ii)] Let $X$ be a  smooth connected  quasi-projective variety with $a\in
X(k)$, such that  $\pi_1^{\rm \acute{e}t}(X,a)$ is trivial. Then is
$\pi_1^{\rm alg }(X,a)$ trivial as well?
\end{itemize}
\end{questions}

If $X$ is smooth, then $X\setminus \Sigma$ has the same $ \pi_1^{\rm
\acute{e}t}$, resp. $  \pi_1^{\rm alg }$ as $X$ when $\Sigma$ has codimension
$\ge 2$. So Questions \ref{q:proj} reduce to $X,Y$ projective
smooth in i), and
$X$ projective smooth in ii) if we start with $X\setminus \Sigma_X$ and
$Y\setminus \Sigma_Y$, with $\Sigma_X, \Sigma_Y$ of codimension $\ge 2$. On the
other hand, by 
blowing up several times and removing divisors of a smooth projective variety with $\pi_1^{\rm \acute{e}t}(X,a)=0$,  one easily constructs examples 
of non-proper smooth varieties $X^0$  with $\pi_1^{\rm \acute{e}t}(X^0,a)=0$, with a nice normal compactification such that the locus at infinity has codimension $\ge 2$, but the compactification is not  smooth (see \cite{Ki}).
So in absence of resolution of singularities, and in view of the difficulty
to find interesting examples, this is meaningful to ask, even if in
characteristic $0$, the answer is negative, as one sees for example on the affine line. Indeed, 
 any non-zero closed $1$ form $\omega$ on $\A^1$ defines a non-trivial connection
$d+\omega$ on $\sO_{\A^1}$. But, in characteristic $p>0$, $\pi_1^{\rm
\acute{e}t}(\A^1,0)$ is highly non-trivial.

When $X$ is smooth projective,  Question \ref{q:tame}, or, equivalently, Question~\ref{q:proj} is Gieseker's conjecture 
\cite[p.~8]{G}. In analogy with  Corollary~\ref{cor:f}, Gieseker \cite{G}, {\it loc.~cit.} further raised the following questions.

\begin{questions} \label{q:abelian}
Let  $X$ be a smooth connected projective variety with $a\in
X(k).$
\begin{itemize}
\item[i)] Does
every irreducible $\sO_X$-coherent $\sD_X$-module
  have
rank $1$ if and only if the commutator 
$[\pi_1^{\rm \acute{e}t }(X,a), \pi_1^{\rm \acute{e}t }(X,a)]$ 
is a pro-$p$-group?
\item[ii)] Is every $\sO_X$-coherent $\sD_X$-module
a direct sum of rank $1$ ones if and only if $ \pi_1^{\rm \acute{e}t }(X,a)$ is abelian without nontrivial $p$-power quotient?
\end{itemize}
\end{questions}

\medskip

Finally,  assuming $X$ projective smooth,  while we have at disposal Langer's quasi-projective moduli varieties of Gieseker semi-stable  bundles \cite{L}, we do not have any analog to Simpson's  quasi-projective moduli space $M_{dR}(X)$ and of $M_{{\rm Betti}}(X)$.

\section{Results} \label{sec:results}
The aim of this section is to describe the answers we have to Questions \ref{q:tame}, \ref{q:proj}, \ref{q:abelian}. Unfortunately, they are only partial answers.

\subsection{Question~\ref{q:proj}, ii)}
It has  a positive answer when $X$ is projective (see \cite[Theorem~1.1]{EM}), this is the
main result. We describe now the analogies between the proof of
Proposition~\ref{f}, ii) over
 $\C$ and the proof of \cite{EM}, {\it{ loc.~cit}}.
 
\medskip

Over $\C$ we were using that the monodromy group of a representation of $\pi_1^{\rm top}(X,a)$ lies in some $GL(r, A)$ for $A$ a ring of finite type over $\Z[\frac{1}{N}]$, as $\pi_1^{\rm top}(X,a)$ is finitely generated. This is a statement on the complex local system, which is not easy to transpose on the side of the $\sD_X$-modules.

\medskip

 In characteristic $p>0$,  we do not have a group of finite type controlling the representation 
$\pi^{\rm alg}(X,a)\to GL(r, k)$ coming from Tannaka theory. Going to the side
of stratified bundles, we do not have quasi-projective moduli spaces for them
either. But, $X$ being projective, a stratified bundle $\E$ is up to isomorphism
determined by the underlying vector bundles $(E^{(n)})_{n\ge 0}$ (Katz' theorem
\cite[Proposition~1.7]{G}). If we can make sure that the $E^{(n)}$ are all
$\mu$-stable of slope $0$, we can ``park''  them all on one quasi-projective
moduli $M$.  That we may assume that they are $\mu$-stable relies on two facts. 
Firstly there is a  structure theorem asserting that there is a natural number
$n_0$ such that the shifted stratified bundle $(E_{n_0}, E_{n_0+1},\ldots),
(\sigma_{n_0}, \sigma_{n_0+1}, \ldots)$ is a successive extension of stratified
bundles, each of which with the property that its underlying vector bundles are
all $\mu$-stable of slope $0$ (\cite[ Proposition~2.3] {EM}. Secondly, 
extensions of the trivial stratified bundle by itself 
form a group, which, again using Katz' theorem, is identified with the group
$\varprojlim_{F} H^1(
X, \sO_X)$, where $F$ is the absolute Frobenius, which is trivial if
$\pi^{\rm{\acute{e}t}}_1(X,a)=0$. In fact 
$H^1(X, \Z/p)=0$ is enough here to conclude
(\cite[Proposition~2.4]{EM}).

\medskip
Further over $\C$, once one has  $\rho: \pi_1^{\rm top}(X,a) \to GL(r, A)$, we find a closed point $s\in \Spec A$ such that the specialization $\rho\otimes \kappa(s): \pi_1^{\rm top}(X,a) \to GL(r, \kappa(s))$ is not trivial if $\rho$ is not trivial. 

\medskip

In characteristic $p>0$, we consider a model $M_R$ of  $M$ over some ring  $R$
of finite type over $\F_p$. Even of $M$ is not a fine moduli space, it is better
than a coarse moduli space, in particular it gives a moduli interpretation of
$\bar{\F}_p$-points of $M_R$, assuming that $X$ itself has a model $X_R$ over
$R$. We consider the Zariski closure $N\subset M $  of all the moduli points
$E^{(n)}$, and specialize $N$  to a closed point $s$ of $\Spec R$ for some $R$
on which it is defined. Then the whole point is to show that some such
specialization contains a moduli point $V$ say which is fixed by some power of
Frobenius (\cite[Theorem~3.14]{EM}). This yields a Lang torsor over
$X_R\otimes_R s$ by resolving the Artin-Schreier type equation $(F^{m}-{\rm
Id})^*(V)=0$ \cite[Satz~1.4]{LS}.  The theory of specialization of the \'etale
fundamental group \cite[X,~Th\'eor\`eme~3.8]{SGA1} forces then $V$ to be
trivial, a contradiction, unless $N$ is empty. Now, in order to find such a
point $V$, one has to apply Hrushovsky's fundamental result
\cite[Corollary~1.2]{H} stemming from model theory. This, undoubtedly, is a
deeper step than finding over $\C$ an $s$ for which $\rho\otimes \kappa(s)$ is
not trivial.

\subsection{Question~\ref{q:abelian}}
It has a positive answer.  The surjection  $\pi_1^{\rm alg}(X,a) \to \pi_1^{\rm \acute{e}t}(X,a)$
together with  some classical representation theory of finite $p$-groups imply the one direction (see \cite[Theorem~1.10]{G}): 
assuming irreducible $\sO_X$-coherent $\sD_X$-modules  have rank $1$, then $[\pi_1^{\rm \acute{e}t}(X,a),\pi_1^{\rm \acute{e}t}(X,a) ]$ is a pro-$p$-group and 
if the category of $\sO_X$-coherent $\sD_X$-modules is semi-simple, then $\pi_1^{\rm \acute{e}t}(X,a)$ is abelian without $p$-power quotient.

\medskip

We described the other direction (see \cite{ES}). The method is derived from \cite{EM}, with some new ingredients, which we briefly explain now.

\medskip

We discuss i). By the discussion of the proof of Question~\ref{q:proj} ii),  we extract the information that if $X$ has a non-trivial $\sO_X$-coherent $\sD_X$-module of rank $\ge 2$, then it also has one of the same rank with finite monodromy \cite[Theorem~2.3]{ES}. This, by classical theory of representation of $p$-groups, forbids the commutator of the monodromy group to be a $p$-group.

\medskip

We discuss ii). One has to show that the category of $\sO_X$-coherent $\sD_X$-modules is semi-simple. As irreducible  $\sO_X$-coherent $\sD_X$-modules have rank $1$, 
one has to show that there are no non-trivial extension of the trivial $\sO_X$-coherent $\sD_X$-modules by a rank $1$ one. 
To this aim, we have to replace Langer's moduli $M$ in the proof of Question~\ref{q:proj} ii) by some quasi-projective moduli $M'$ say,  to be constructed,  of non-trivial
extensions of $\sO_X$ by line bundles $L$.  The assumption implies that such an extension of $\sO_X$ by a torsion line bundle $L$ splits, and does after specialization to  $X_R\otimes_R s$ for a closed point.  Hrushovsky theorem as explained before enables one  to find split extensions as moduli points of $N'\otimes_R s$, a contradiction.

\subsection{Rank $1$} 
On $X$ quasi-projective smooth over an algebraic closed field $k$ of characteristic $p>0$, $\sO_X$-coherent $\sD_X$-modules of rank $1$
 are always regular singular \cite[Theorem~3.3]{G}. So Question~\ref{q:tame} ii) and Question~\ref{q:proj} ii) are equivalent for rank $1$ objects. It has a positive answer \cite[Section~5]{Ki}. In fact, if one restricts to rank $1$, it is enough to assume that $\pi_1^{{\rm \acute{e}t, ab}, p'  }(X,a)=0$, the prime to $p$ maximal abelian quotient of $\pi_1^{\rm \acute{e}t}(X,a)$. The point is that under this assumption, necessarily units on $X$ are constant. This implies that a rank $1$ stratified bundle is entirely determined by the underlying bundles $(L^{(n)})_{n\in \N}$. Then one shows by geometry that the assumption implies that $\Pic(X)$ is a finitely generated abelian group (\cite[Section~5]{Ki}).

\subsection{The affine space.}
A simple discussion (led with Lars Kindler) implies that Question~\ref{q:tame}
ii) has positive
answer for $X=\A^n$.
Indeed by Theorem~\cite[Theorem~5.3]{G}, any flat  regular singular $\sO_X$-coherent
$\sD_X$-module is a sum of such  rank $1$ $\sD_X$-modules. If $L$ is such  a rank $1$ $\sD_X$-module, then 
 there is a $\sO_X$-locally free
extension $\hat{L}$ on $\P^n$ on which the action of $\sD_X$ on $L$ extends to
an action of $\sD_{\P^n}(\log \infty)$. As an algebraic vector bundle,
$L=\sO_{\P^n}(d\cdot \infty)$, where $\infty=\P^n\setminus X$.  The connection
$\nabla^{(0)}: \hat{L}\to
\Omega^1_{\P^n}(\log \infty)\otimes \hat{L}$ on it is necessarily of the shape 
$d+A$. 
Here $d$ is the  connection which is uniquely defined  on
$\sO(d\cdot \infty)$ by its restriction to $\sO_{\P^n}
\hookrightarrow
\sO(d\cdot \infty)$  given by the section $d\cdot \infty$, where it is
defined
by $d(1)=0$. Then $A\in H^0(\P^n, \Omega^1_{\P^n}(\log \infty))$, which is $0$. 
  So finally $\nabla^{(0)}$ has no poles and is trivial.
We repeat this argument for $\hat{L}^{(i)}$ on $(\P^n)^{(i)}$ and conclude 
that in fact $\hat{L}$ is a $\sO_{\P^n}$-coherent $\sD_{\P^n}$-module, which is
trivial.

\section{Some comments and questions}
\subsection{Moduli}
On $X$ projective smooth over an algebraic closed field $k$, it would be very nice to extend Simpson's theory to construct  moduli spaces (would they be quasi-projective?) for stratified bundles and for $\sO_X$-coherent $\sD_X$-modules. 

\medskip

Recall that over a field of characteristic $0$, a bundle with a flat connection  $(E,\nabla)$ with finite monodromy is uniquely  determined  by $E$ (\cite{EH}, {\it loc.~cit.}).
 It would be very nice to understand how the homeomorphism between $M_{{\rm Betti}}(X)$ and $M_{dR}(X)$, followed by the forgetful morphism $(E,\nabla)\mapsto E$ with values in the moduli of bundles,  transports points corresponding to finite local systems.
 
 \subsection{From characteristic $p$ to characteristic $0$ and vice-versa}
 In spite of the strong analogy between Proposition~\ref{f} ii) for $X$ projective and Question~\ref{q:tame} ii) for $X$ projective (\cite{EM}, {\it loc.~cit.}), and between Corollary~\ref{cor:f} and Question~\ref{q:abelian} (\cite{ES}, {\it loc.~cit.}), we do not have a direct way to go back and forth between characteristic $p>0$ and characteristic $0$. Indeed, a flat connection in characteristic $0$ specializes to a flat connection in characteristic $p>0$, that is to an action of the differential operators of order $\le 1$, but this action does not extend to an action of $\sD_X$.  Already to request $p$-curvature $0$ 
  (which is equivalent, as already explained,  to the existence of $E^{(1)}$)
 for almost all $p$ should be, according to Grothendieck's  $p$-curvature conjecture, equivalent to the connection in characteristic $0$ to have finite monodromy. 
 Assuming Simpson's moduli spaces $M_{dR}(X)$ specialize to moduli spaces of flat connections for almost all $p$, 
 at least at the level of geometric points, a positive answer to Grothendieck's conjecture, for $X$ projective, would provide some answer to the previous question. In the present state of knowledge, we know  the well behavior of moduli points corresponding to finite monodromy 
 only in equal characteristic $0$ for de Rham moduli  (\cite{An}) and $p>0$  for moduli of semi-stable sheaves or, in absence of moduli, for families of stratified bundles, assuming the monodromy groups have order prime to $p$  (\cite{EL}). 
 
 \subsection{Question~\ref{q:tame} ii)} Of course, all unanswered questions in \ref{q:proj}, \ref{q:tame} are of interest, that is all relative cases and the absolute non-proper case. Focusing on this one, 
 even if one assumes that  $X$ has a compactification $\hat{X}$ such that $\hat{X}\setminus X$ is a strict normal crossings divisor,  the method of proof in the projective case  \cite{EM} can  not be overtaken as such.  The bundles $\hat E^{(n)}$ on $\hat X$ do not all have the same Chern classes, so we can not park them all in one moduli space. 
 
 \subsection{Finiteness} A natural question which comes from \cite{BKT}, and 
from \cite{ES}, is whether or not, if $X$ is smooth projective over an
algebraically closed field, such that  the profinite completion map $\pi_1^{{\rm
alg}}(X,a)\to \pi_1^{{\rm \acute{e}t}}(X,a)$ is an isomorphism, i.~e. all
$\sO_X$-coherent $\sD_X$-modules have finite monodromy, $\pi_1^{{\rm
\acute{e}t}}(X,a)$ itself is finite. 

\bibliographystyle{plain}

\begin{thebibliography}{99}
\bibitem {An} Andr\'e, Y.: {\em Sur la conjecture des $p$-courbures de
Grothendieck-Katz et un probl\`eme de
Dwork}, Geometric aspects of Dwork theory, Vol. I, II, 55--112,
Walter de Gruyter GmbH \& Co. KG, Berlin, 2004.
\bibitem{AH} Atiyah, M., Hirzebruch, F.: {\em Analytic cycles on complex manifolds},  Topology {\bf 1} (1962), 25--46.
\bibitem{BKT} Brunebarbe, Y., Klingler, B., Totaro, B.:  {\em Symmetric
differentials and the fundamental group of projective manifolds}, preprint 2012.
\bibitem{CS} Corlette, K., Simpson, C.: {\it On the classifications of rank-two
representations of quasi-projective fundamental groups}, 
Compositio math. {\bf 144}  (2008), 1271--1331. 
\bibitem{Del} Deligne, P.: {\em \'Equations Diff\'erentielles \`a Points Singuliers R\'eguliers}, Lecture Notes in Mathematics {\bf 163} (1970), Springer Verlag.
\bibitem{Del2} Deligne, P.:  {\em Le groupe fondamental de la droite projective moins trois points}, in ``Galois Groups over $\Q$'', Math. Sc. Res. Inst. Publ. {\bf 16},  79--297, Springer Verlag.
\bibitem{DM} Deligne, P., Milne, J.:  {\em Tannakian categories}, Lecture Notes in Mathematics {\bf 900}, Springer Verlag.
\bibitem{dS} dos Santos, J.-P.-S.: {\em Fundamental group schemes for stratified sheaves},  J. Algebra  {\bf 317}  (2007),  no. 2, 691--713.
\bibitem{EH}  Esnault, H., Ph\`ung H\^o Hai:  {\em The fundamental groupoid scheme and 
applications},   Annales de l'Institut Fourier  {\bf 58} (2008), 2381--2412.
\bibitem{EM} Esnault, H., Mehta, V.: {\em Simply connected projective manifolds in characteristic $p>0$ have no nontrivial stratified bundles},  Inventiones math. {\bf 181} (2010), 449--465.
\bibitem{ES} Esnault, H., Sun, X.:  {\em Stratified bundles and \'etale fundamental group}, preprint
2011, 18 pages, {\tt  http://www.esaga.uni-due.de/helene.esnault/downloads/} [105].
\bibitem{EL} Esnault, H., Langer, A.:  {\em On a positive equicharacteristic variant
of the
$p$-curvature conjecture}, preprint 2011, 32 pages, {\tt http://www.esaga.uni-due.de/helene.esnault/downloads/} [104].
\bibitem{G} Gieseker, D.: {\em Flat vector bundles and the fundamental group in non-zero characteristics}, Ann. Scu. Norm. Sup. Pisa,  
4. s\'erie, tome {\bf 2}, no 1 (1975), 1--31. 
\bibitem{Gr} Grothendieck, A.: {\em Repr\'esentations lin\'eaires et compactifications profinies des groupes discrets}, Manuscripta mathematica, {\bf 2} (1970), 375--396.
\bibitem{H} Hrushovski, E.: {\em The Elementary Theory of Frobenius Automorphisms}, {\tt http://de.arxiv.org/pdf/math/0406514v1}.
\bibitem{Ka} Katz, N.: {\em Nilpotent connections and the monodromy theorem},
Publ. math. I.H.\'E.S. {\bf 39} (1970), 175--232.
\bibitem{KS} Kerz, M., Schmidt, A.:  {\em On different notions of tameness in arithmetic geometry}, Math. Ann. {\bf 346} 3 (2010), 641--668.
\bibitem{Ki} Kindler, L.: PhD Thesis, Essen 2012.
\bibitem{Kl} Klingler, B.: {\em Symmetric differentials, K\"ahler groups and
ball quotients}, preprint 2012
{\tt http://people.math.jussieu.fr/\~{}klingler/papers.html}.
\bibitem{LSe} Lang, S., Serre, J.-P.:  {\em Sur les rev\^etements non ramifi\'es des vari\'et\'es alg\'ebriques},  Am. J. of Maths {\bf 79} no 2  (1957), 319--330.
\bibitem{LS} Lange, H., Stuhler, U .: {\em Vektorb\"undel auf Kurven und Darstellungen der algebraischen Fundamentalgruppe}, Math. Zeit. {\bf 156} (1977), 73--83.
\bibitem{L} Langer, A.: {\em Semistable sheaves in positive characteristic}, Annals of Maths. {\bf 159} (2004), 251--276.
\bibitem{LO} Lieblich, M., Olsson, M.: {\em Generators and relations for the
\'etale fundamental group},  Pure Appl. Math. Q. {\bf 6} 1 (2010), 209--243,
Special Issue in Honor of J. Tate, Part 2 .
\bibitem{Mal} Malcev, A.: {\em On isomorphic matrix representations of infinite groups}, Rec. Math. [Mat. Sbornik] N.S.  {\bf 8} (50)  (1940), 405--422.
\bibitem{No} Nori, M.: {\rm The fundamental group scheme}, Proc. Indian Acad. Sci. {\bf 91} (1982), 73--122.
\bibitem{OV}  Orgogozo, F., Vidal, I.: {\em Le th\'eor\`eme de sp\'ecialisation du groupe fondamental}, in ``Courbes semi-stables et groupe fondamental en g\'eom\'etrie alg\'ebrique'' Prog.   Math. {\bf 187} (2000), 169--184.
\bibitem{Or} Orgogozo, F.: {\em Alt\'erations et groupes fondamentaux premiers \`a $p$},  Bull. SMF {\bf 131} 1 (2003), 123--147.
\bibitem{Po} Poincar\'e, H. : {\em Analysis Situs},  Journal de l'\'Ecole
Polytechnique s\'er. 2, {\bf 1} (1895), 1--123.
\bibitem{Simp} Simpson, C.: {\em Moduli of representations of the fundamental
group of a smooth projective variety}, Publ. math. I.H.\'E.S. {\bf 79} (1994),
47--129.
\bibitem{SGA1} SGA1.: {\em Rev\^etements \'etales et groupe fondamental}, SGA 1.


\end{thebibliography}
\renewcommand\refname{References}

\end{document}